\documentclass[a4paper,11pt]{amsart}
\usepackage[left= 3.4cm,right = 3.4cm, bottom=4.5cm, top=3.5cm]{geometry}
\title{Auxiliary Hessian Equations}

\author{Johannes G.\ Scheffler}

\usepackage[utf8]{inputenc}
\usepackage[ngerman, english]{babel}
\usepackage[autostyle=true,german=quotes]{csquotes}
\usepackage[T1]{fontenc}
\usepackage{graphicx, subfig}
\graphicspath{{img/}}
\usepackage{lmodern}
\usepackage{color}
\usepackage[backend=biber,maxbibnames=99,style=alphabetic]{biblatex}  
\usepackage{newunicodechar}
\newunicodechar{∞}{$\infty$}
\usepackage{hyperref,mathtools}

\usepackage{verbatim} 

\usepackage{amsfonts}
\usepackage{amsmath}
\usepackage{amsthm} 

\usepackage{amssymb}
\usepackage[new]{old-arrows}
\usepackage{faktor}
\usepackage{tikz-cd} 

\usepackage{mathrsfs}

\usepackage{relsize}

\usepackage{ dsfont }

\newcommand{\RR}{\mathbb{R}}
\newcommand{\CC}{\mathbb{C}}

\newcommand{\dbar}{\overline{\partial}}

\newcommand{\Cont}{\mathscr{C}}

\newcommand{\ii}{\mathrm{i}}

\usepackage{enumitem}

\usepackage{aligned-overset}

\begin{document}

\theoremstyle{definition}
\newtheorem{Definition}{Definition}
\newtheorem{Beispiel}[Definition]{Example}
\newtheorem{Construction}[Definition]{Construction}
\newtheorem{Assumption}[Definition]{Assumption}

\theoremstyle{remark}
\newtheorem{Bemerkung}[Definition]{Remark}

\theoremstyle{plain}
\newtheorem{Satz}[Definition]{Satz}
\newtheorem{Folgerung}[Definition]{Corollary}
\newtheorem{Lemma}[Definition]{Lemma}
\newtheorem{Proposition}[Definition]{Proposition}
\newtheorem{Theorem}[Definition]{Theorem}
\newtheorem{Conjecture}[Definition]{Conjecture}
\newtheorem*{THeorem}{Theorem} 
\newtheorem*{SAtz}{Satz}

\maketitle



\pagenumbering{arabic} 


\enlargethispage{1\baselineskip}

\begin{abstract}
     We propose $k$-Hessian equations as auxiliary equations for degenerations of Kähler classes to nef classes of numerical dimension $k$. As proof-of-concept results, we provide a mean-value inequality and a priori estimates for Hessian equations.
\end{abstract}

\section{Key words}

\section{Introduction}
\subsection{Idea of the paper} Inspired by the recent gigantic progress obtained by the systematic use of auxiliary Monge-Amp\`ere equations (started by Guo-Phong-Tong in \cite{guo$L^infty$EstimatesComplex2021}; cf.\ the surveys \cite{guoAuxiliaryMongeAmpereEquations2022} and \cite{guoGeometricInvariantsMongeAmpere2025} and the references therein) we want to propose other non-linear partial differential equations suited as auxiliary equations, namely complex Hessian equations. Similar to the Monge-Amp\`ere equation, the Hessian equation -- in a local setting -- prescribes properties of the eigenvalues of the complex Hessian matrix of a function. Instead of the product of all the eigenvalues, however, we consider more generally elementary symmetric polynomials of degree $1< k\leq n$. There is also a well-developed existence theory of solutions to Hessian equations, established in the last two decades by S.-Y.\ Li, Z. B\l ocki, Z. Hou, A. Jbilou, Hou-Ma-Wu, Dinew-Ko\l odziej, and G. Székelyhidi (see \cite{liDIRICHLETPROBLEMSSYMMETRIC2004} \cite{blockiWeakSolutionsComplex2005} \cite{houComplexHessianEquation2009} \cite{jbilouEquationsHessiennesComplexes2009} \cite{houSecondOrderEstimate} \cite{dinewLiouvilleCalabiYauType2017} \cite{szekelyhidiFullyNonlinearElliptic2018}). An important advantage of Hessian equations is the dependence of certain cohomological constants only on $k$-th powers of a class, allowing a lower bound for a nef class being only of numerical dimension $k$ instead of big. This is closely tied to the larger space of admissible solutions -- so-called $k$-subharmonic functions -- having relaxed positivity requirements compared to psh-functions.

Naturally, one needs to pay a price in the form of weaker properties, for example no integral bound for the exponential of solutions (the $\alpha$-invariant estimate) is available but only $L^q$ bounds for certain exponents. Therefore, modifications to Guo-Phong-Tong's technique are required. The general structure of the argument stays very similar, though:
\begin{enumerate}
    \item Solve a Hessian equation with a quantity of interest, e.g., subsolutions for degenerating linear elliptic operators, on the RHS.
    \item Compare the solution $\psi$ of the auxiliary Hessian equation with the quantity of interest. For this, apply the maximum principle with a suitable linear elliptic operator $L$ and use the mixed Hessian inequality -- as a replacement for the AM-GM inequality -- to compare $L\psi$ and the RHS of the auxiliary equation.
    \item From the uniform $L^p$-bound of $k$-sh functions derive an integral bound for the RHS of the auxiliary equation, i.e., the quantity of interest.
\end{enumerate}

These uniform integral bounds can then be exploited, e.g., in a De Giorgi iteration scheme.

To illustrate this approach, first, we obtain \textit{a priori} estimates for the $k$-Hessian equation in a ray of the Kähler cone degenerating to a nef class of numerical dimension $k$. Secondly, we will prove a mean value inequality in the style of Lemma 2 of \cite{GPS22} for a suitably degenerating linear elliptic operator, namely the linearized Hessian operator.

\subsection{Setup and main results}
    Let $(X,\omega_X)$ be a $n$-dimensional compact Kähler manifold with reference Kähler metric $\omega_X$. Assume $\mathrm{Vol}(X, \omega_X)=1$. Let $\chi$ be a closed, smooth, real $(1,1)$-form such that its class $[\chi]\in H^{1,1}(X, \RR)$ is nef and has numerical dimension $1 < \nu([\chi]) = k < n$. Let us assume for simplicity that the form $\chi$ is in addition semi-positive. This can be avoided by modifying the arguments similarly to the work of Guo-Phong-Tong-Wang (cf.\ \cite{GPTW24}, also \cite{schefflerAuxiliaryMongeAmpereEquations2024}) using the $k$-sh envelope of $\chi$. Let $\omega_t \coloneqq \chi + t \omega_X$ for $t\in (0,1]$ and note that this is a Kähler form. 
    Let
     $V_{k,t} \coloneqq \int_X \omega_t^k \wedge \omega_X^{n-k}$ is the intersection number playing the role of a \textit{mixed volume}. Note that $V_{k,0} = \lim_{t\rightarrow 0} V_{k,t} >0$ by the assumption on the numerical dimension of $[\chi]$.

    \begin{Assumption}
        Let the numerical dimension of $[\chi]$ be $n-1$, i.e., $k=n-1$. 
    \end{Assumption}

    This assumption is necessary due to the range of exponents for which an uniform $L^p$-bound for $k$-sh functions is known (by Lemma 8 of \cite{guo$L^infty$EstimatesComplex2021} for $p\in ]0, \frac{n}{n-k}[$). Conjecturally, this range is $p\in ]0, \frac{nk}{n-k}[)$. For us, the key ingredient is an uniform $L^p$-bound for some $p > k$. So if the conjectural bound can be established, the following statements hold for all $1<k<n$. We will use in the proofs $k$ instead of $n-1$ to indicate how the numerical dimension affects the arguments. Also, up to Lemma \ref{estimate_modified_energy} this restriction on $k$ is not needed.

    \subsubsection{Uniform $L^\infty$-estimates for degenerating Hessian equations}
    We will consider the following family of Hessian equations, degenerating as $t\rightarrow 0$. Let $F_t$ be given smooth function, normalized so that $\int_X e^{F_t}\omega_X^n = \int_X \omega_t^k \wedge \omega_X^{n-t}$. 
    Similarly to the technique of auxiliary Monge-Ampère equations, an entropy bound is required. 

    \begin{Assumption}
        Let for some $2<q<\infty$ and all $t\in ]0,1]$ the $L^q$-norm of $e^{F_t}$ with respect to $\omega_X^n$ be bounded by some uniform constant $N>0$, i.e.,
        \begin{align}
            \int_X e^{qF_t} \omega_X < N.
        \end{align}
    \end{Assumption}

    For general $k$ we need -- assuming the conjectural range of uniform $L^p$ bound for $k$-sh functions -- this entropy condition for some $q > \max\{1,\frac{n-k-k^2}{k^2}\}$.

    \begin{Theorem} \label{Thm_HEq_Linftyestimate}
         Assume that for $t>0$ we have solutions $\varphi_t$ for the Hessian equations
         \begin{align} \label{HEq_Linftyestimates}
            (\omega_t + dd^c \varphi_t )^k \wedge \omega_X^{n-k} = e^{F_t} \omega_X^n, \quad \sup_X \varphi_t = 0.
         \end{align}
         Then there is a constant $C=C(X, \omega_X, \chi, N, q) >0$ independent of $t$ such that 
         \begin{align*}
             \sup_X |\varphi_t| \leq C.
         \end{align*}
    \end{Theorem}

    \begin{Bemerkung}
        By the solution theory to Hessian equations (see Proposition 21 in \cite{szekelyhidiFullyNonlinearElliptic2018} and Section \ref{subsec_CplxHessianEquation}), such $\varphi_t$ do exist and are unique. 
    \end{Bemerkung}

    The first $L^\infty$-estimates for Hessian equations with degenerating background metric were obtained by Guo-Phong-Tong (cf.\ Theorem 2 and Theorem 6 in \cite{guo$L^infty$EstimatesComplex2021}). Since they consider a very general set-up for fully non-linear equations, their normalizations differ from ours. Moreover, in their case two constants need to be controlled, in particular the volume induced by the degenerating Kähler metric in the denominator. Therefore, they obtain the better estimates -- thanks to the superior properties of psh functions -- only in the case that $[\chi]$ is big. In the follow-up work \cite{GPTW24} of Guo-Phong-Tong-Wang, this condition is dropped; however, it is not explained how to avoid the volume in the denominator. Guedj-Lu obtain $L^\infty$ estimates in the Hermitian setting assuming $\chi$ (in their notation $\omega$, see Theorem B in \cite{guedjDegenerateComplexHessian2025}) is bounded by a semi-positive big form. They only need an Orlicz-norm control on the RHS of the Hessian equation. Using envelopes, their method is completely distinct from ours. W. Sun obtains in \cite{sunWeakSolutionsComplex2024} estimates via local auxiliary Monge-Ampère equations, i.e., solving the Dirichlet problem in coordinate balls. Degenerations are considered, but are bounded from below by a semi-positive and big form. Y. Qiao considers in \cite{qiaoSharp$$textrmL^infty$$estimates2025} again a very general set-up for fully non-linear equations and obtains sharp $L^\infty$-estimates, combining the auxiliary Monge-Ampère technique with a Hölder-Young estimate, and a new gluing result for certain psh functions. For degenerations, there is again a dependence on the volume of the degenerating form. Finally, we want to point out the work of Pang-Sun-Wang-Zhou (cf.\ \cite{pangDegenerateComplexHessian2025}) where the Hessian equation is resolved for $\chi$ merely $(\omega_X, k)$-positive. However, uniform estimates for degenerations are only obtained under a sub-slope condition, which in the Kähler case translates to positive volume. 

    \subsubsection{A mean-value inequality}
    
    A very useful tool established via auxiliary Monge-Ampère equations is a mean value inequality (see Lemma 2 in \cite{GPS22} by Guo-Phong-Sturm, extended in \cite{GPSS24}; moreover, there is a family version by Guedj-Tô, Lemma 2.1 of \cite{guedjKahlerFamiliesGreens2025}) for the Laplacian associated to the degenerating Kähler metric. It was used to derive estimates for the associated Green's function leading to diameter estimates and non-collapsing results. These turned out to be key tools in the study of singular varieties via resolution, for example to obtain Hermite-Einstein metrics and study the Bogomolov-Gieseker inequality (cf.\ \cite{CGNPPW23} \cite{guenanciaBogomolovGiesekerInequality2025} \cite{fuOrbifoldBogomolovGiesekerInequalities2025} and C.-M.\ Pan \cite{panGauduchonMetricsHermiteEinstein2025} for the non-Kähler setting).

    Instead of the Laplacian (the linearization of the Monge-Ampère), we consider here the linearization of the Hessian operator $u \mapsto (\omega_t + dd^c u)^k \wedge \omega_X^{n-k}$ at zero,
    \begin{align*}
            L^{n-1}_{\omega_t, \omega_X} f = \frac{dd^c f \wedge\omega_t^{n-2}\wedge \omega_X}{\omega_t^{n-1} \wedge \omega_X}.
    \end{align*}
    It is a second order linear elliptic operator, however the ellipticity might degenerate as $t\rightarrow0$.
    In the case $\omega_t = \omega_X$ of one fixed Kähler form, $L^{n-1}_{\omega_t, \omega_X} = C \Delta_{\omega_t}$. 
    
    Let
    \begin{align*}
        F_t \coloneqq \log \left( \frac{1}{V_{k,t} }\frac{\omega_t^k \wedge \omega_X^{n-k}}{ \omega_X^n} \right)
    \end{align*}
    and assume the following entropy-style bound as above:

    \begin{Assumption}
        Let for some $2<q<\infty$ and all $t\in ]0,1]$ the $L^q$-norm of $e^{F_t}$ with respect to $\omega_X^n$ be bounded by some uniform constant $N>0$, i.e.,
        \begin{align}
            \int_X e^{qF_t} \omega_X < N.
        \end{align}
    \end{Assumption}

    \begin{Theorem} \label{meanvalueineq}
    Let $a > 0$ be fixed. There is a constant $C > 0$ depending only on $n$, $q$, $\chi$, $\omega_X$, $N$, and $a$ such that all $v \in L^1(X, \omega_t^{n-1}\wedge \omega_X)$ which
    \begin{enumerate}
        \item are twice continuously differentiable on $\Omega_{-c} \coloneqq \{v > -c\}$ for some $c > 0$,
        \item satisfy $\int_X v \omega_t^{n-1}\wedge \omega_X = 0$, and
        \item whose mixed Laplacian $L^{n-1}_{\omega_t, \omega_X}$ with respect to $\omega_t$ and $\omega_X$ is bounded in $\Omega_0$ from below by
            \begin{align}
                L^{n-1}_{\omega_t, \omega_X} v \geq -a
            \end{align}
    \end{enumerate} 
    are bounded from above by
    \begin{align}
        \sup_X v \leq C(1 + \|v\|_{L^1(X, \omega_t^{n-1}\wedge \omega_X)}).
    \end{align}
\end{Theorem}

\subsection{Declaration on the use of artificial intelligence}
    The main results of this paper, Theorem \ref{Thm_HEq_Linftyestimate} and Theorem \ref{meanvalueineq}, and their proofs presented here were obtained prior to November of 2025 and without the use of large language models nor other forms of artificial intelligence. In the process of writing this article, however, AI tools like the ChatGPT models available in the "Plus" subscription where used to find references and precise statements of background results, and clarify contributions by other authors. Moreover, AI was used to help the author manage the writing process. The author verified statements made by AI to the best of this knowledge. Nevertheless, errors may have occured and these rest in the full responsibility of the author.

\subsection{Acknowledgements} First, I would like to thank my advisor, Mihai P\u{a}un, for many encouraging and helpful discussions on the topic in this paper. Most importantly, for giving me the freedom to develop/follow through on this idea. Also, my colleagues at Bayreuth helped a lot, especially on preliminary efforts to generalize the AM-GM inequality (though finally neither successful not necessary). Next, I want to express my gratitude to Vincent Guedj for a very fruitful conversation where he pointed out several important references on the solvability of Hessian equations. The basic idea of this paper was developed during a stay at the Department of Mathematical Sciences at Gothenburg. I would like to thank the institution and the Complex Geometry group, in particular David Wytt-Nyström, for their hospitality and the motivating working athmosphere. Furthermore, I want to acknowledge the financial support received by the Studienstiftung des deutschen Volkes and the Marianne-Plehn Programm of the Elitenetzwerk Bayern.

\section{Background results}
    We collect here the results on $k$-positive functions and Hessian equations used throughout the paper, as these are less well-known than the corresponding results for psh functions and Monge-Ampère equations.

    \subsection{$k$-positive forms and $k$-subharmonic functions}
     \begin{Definition}
         Let $\alpha$ be a smooth real $(1,1)$-form. We say that $\alpha$ is $(\omega_X, k)$-positive if for all $1\leq j \leq k$ we have
         \begin{align*}
             \alpha^j \wedge \omega_X^{n-j} \geq 0. 
         \end{align*}
         Moreover, $\alpha$ is called strictly $(\omega_X, k)$-positive if for all $1\leq j \leq k$ and all continuous positive functions $\gamma$ there is an $\varepsilon>0$ such that
         \begin{align*}
             \alpha^j \wedge \omega_X^{n-j} \geq \varepsilon\gamma \omega_X^n. 
         \end{align*}
         We say that a smooth function $u$ is $(\alpha, \omega_X, k)$-subharmonic if the form $\alpha + \ii \partial \dbar u$ is $(\omega_X, k)$-positive, that is for $1\leq j \leq k$ we find
         \begin{align*}
             (\alpha + \ii \partial \dbar u)^j \wedge \omega_X^{n-j} \geq 0.
         \end{align*}
         Analogously, we call $u$ strictly $(\alpha, \omega_X, k)$-subharmonic if the form $\alpha + \ii \partial \dbar u$ is strictly $(\omega_X, k)$-positive.
     \end{Definition}

     \begin{Bemerkung}
         A smooth function $u$ is $(\alpha, \omega_X, k)$-subharmonic if and only if the eigenvalues $(\lambda_1, \dots, \lambda_n)$  of the form $\alpha + \ii \partial \dbar u$ with respect to $\omega_X$ lie in the $k$-positive cone $\Gamma_k$ defined via elementary symmetric polynomials:
         \begin{align*}
             \Gamma_k \coloneqq \{ (\lambda_1, \dots, \lambda_n) \in \RR^n \mid \text{ for } 1\leq j\leq k: \sigma_j (\lambda_1, \dots ,\lambda_n) \geq 0\}
         \end{align*}
         where
         \begin{align*}
             \sigma_j (\lambda_1, \dots, \lambda_n) \coloneqq \sum_{1\leq i_1 \leq \dots \leq i_j \leq n} \lambda_{i_1} \cdot \dots \cdot \lambda_{i_j}.
         \end{align*}
         There are ways to generalize the notion of $k$-positivity to less regular functions (see for example \cite{blockiWeakSolutionsComplex2005} \cite{dinewPrioriEstimatesComplex2014} \cite{guedjDegenerateComplexHessian2025}); however, this is not needed in the following.
     \end{Bemerkung}
    
\subsection{The complex Hessian equation} \label{subsec_CplxHessianEquation}
    Solving a Laplace equation $\Delta_{\omega_X} u = f$ can be understood as prescribing the sum of the eigenvalues of $\ii \partial \dbar u$ with respect to $\omega_X$. Similarly, solving a Monge-Ampère equation can be seen as prescribing the product of all the eigenvalues. It seems natural to interpolate between these two cases by looking at other elementary symmetric polynomials in the eigenvalues, done as follows:
    Let $\alpha$ be a Kähler metric and $f$ a smooth positive function satisfying the compatibility condition $\int_X f\omega_X^n = \int_X \alpha^k \wedge \omega_X^{n-k}$. The complex Hessian equation is given as
    \begin{align} \label{Hessian_equation}
        (\alpha + \ii \partial \dbar \varphi)^k \wedge \omega_X^{n-k} = f \omega_X^n
    \end{align}
    where the unknown $\varphi$ comes from the space of admissible solutions which are the smooth, strictly $(\alpha, \omega_X, k)$-sh functions. 

    \begin{Proposition}[\cite{szekelyhidiFullyNonlinearElliptic2018}] \label{resolution_result_HEq}
        In the above situation, equation (\ref{Hessian_equation}) admits a smooth, strictly $(\alpha, \omega_X, k)$-sh solution, which is unique up to constants.  
    \end{Proposition}

    This statement follows directly from Proposition 21 in \cite{szekelyhidiFullyNonlinearElliptic2018} because the constant function $0$ is a subsolution as $\alpha$ is Kähler. G.\ Székelyhidi's work is in the Hermitian setting and also applies to other fully non-linear equations. It is in some sense the culmination of a two decades lasting pursuit to solve complex Hessian equations: building on the real theory of Caffarelli-Nirenberg-Spruck (cf.\ \cite{caffarelliDirichletProblemNonlinear1985}), S.-Y.\ Li solved the Dirichlet problem in $\CC^n$ under suitably smooth assumptions (cf.\ \cite{liDIRICHLETPROBLEMSSYMMETRIC2004}). Z. B\l ocki generalized Bedford-Taylor pluripotential theory to the Hessian case, introducing the non-smooth $k$-sh functions and weak Hessian operators (cf.\ \cite{blockiWeakSolutionsComplex2005}). First existence theorems in the compact Kähler case were obtained under a non-negativity condition on the holomorphic bisectional curvature by Z.\ Hou (cf.\ \cite{houComplexHessianEquation2009}) and A.\ Jbilou (cf. \cite{jbilouEquationsHessiennesComplexes2009}). To remove the curvature condition, the second-order estimate by Hou-Ma-Wu in \cite{houSecondOrderEstimate} combined with a gradient estimate of Dinew-Ko\l odziej in \cite{dinewLiouvilleCalabiYauType2017} were the key results establishing smooth solvability in the Kähler case for $\alpha = \omega_X$.

    \begin{Bemerkung}
        This is one of the simplest cases of Hessian type equations and what will be of interest in this paper. In general, one can consider the equation in domains in $\CC^n$ with boundary conditions, in the Hermitian (i.e., non-Kähler setting) or with various forms of degenerations, e.g., allowing the RHS to be non-smooth, only semi-positive or the form $\alpha$ to be less positive (see for example \cite{szekelyhidiFullyNonlinearElliptic2018} \cite{dinewPrioriEstimatesComplex2014} \cite{LuNguyen15} \cite{kolodziejWeakSolutionsComplex2016} \cite{guedjDegenerateComplexHessian2025}).  
    \end{Bemerkung}

\subsection{Uniform boundedness of $k$-subharmonic functions in $L^p$}
A key step in the  auxiliary Monge-Ampère technique is the application of the Hörmander-Tian $\alpha$-invariant estimate for $\omega_X$-psh functions, which can be applied uniformly to all functions generated as solution to the Monge-Ampère equations. We use instead the $L^p$ bound for $(\omega_X, \omega_X, k)$-sh functions stated by Guo-Phong-Tong:

            \begin{Proposition}[Lemma 8 of \cite{guo$L^infty$EstimatesComplex2021}] \label{ksh_Lp_bound}
                For any $p\in ]0, \frac{n}{n-k}[$ there is a uniform constant $C = C(n,p,\omega_X)>0$ such that for all smooth $(\omega_X, \omega_X, k)$-sh functions having $\sup_X u = 0$ we have
                \begin{align}
                    \|u\|_{L^p(X, \omega_X^n)} \leq C.
                \end{align}
            \end{Proposition}

            \begin{Bemerkung}
                If the range of exponents can be improved to be $]0, \frac{kn}{n-k}[$ (or any $p>k$), our argument works for all $1<k< n$.
            \end{Bemerkung}

\subsection{The mixed Hessian inequality}

 When deploying auxiliary Monge-Ampère equations, one uses the inequality between arithmetic and geometric mean to estimate the sum of the eigenvalues of a form  $\alpha + \ii \partial \dbar u$ against their product, i.e.,  $(\alpha + \ii \partial \dbar u)^n$. This allows to compare the Laplacian $\Delta_{\alpha} u $ to the Monge-Ampère. Here, we use instead the following mixed Hessian inequality which follows from G\aa rding's inequality for hyperbolic polynomials (see Theorem 5 in \cite{Garding59} which was applied by Z. B\l ocki to $k$-sh functions and Hessian equations in \cite{blockiWeakSolutionsComplex2005}):
    
\begin{Proposition}[statement taken from \cite{dinewPrioriEstimatesComplex2014}] \label{mixed_hessian_ineq}
                Let $U\subset \CC^n$ be open and let $\beta$ be a Kähler metric on $U$. Assume that there are $k$ two times continuously differentiable and $\beta$-$k$-sh functions $u_1, \dots, u_k$ fullfilling the equations
                \begin{align*}
                    (dd^cu_i)^k \wedge \beta^{n-k} = f_i \beta^n
                \end{align*}
                for positive functions $f_1, \dots, f_k$. Then
                \begin{align*}
                    dd^cu_1 \wedge \dotso \wedge dd^cu_k \wedge \beta^{n-k} \geq (f_1 \cdot \dotso \cdot f_k)^{1/k}\beta^n
                \end{align*}
                holds.
\end{Proposition}

%
%
%

\section{Proof of the $L^\infty$-estimate}

 Fix some $t \in ]0,1]$. We need to show that $\sup_X -\varphi_t \leq C$ where $C$ is independent of $t$. This is equivalent to proving that the function
    \begin{align*}
        \phi \colon \RR_{\geq 0} \longrightarrow \RR_{\geq 0}, \quad \quad s \longmapsto   \int_{\Omega_s} e^{F_t} \omega_X^n
    \end{align*}
    vanishes for $s \geq C$. Here and in the following, $\Omega_s = \{-\varphi_t > s\}$ denotes the upper-level set of $-\varphi_t$. The main part of the proof is to establish for every $s\geq 0$ and every $r \geq 0$ the functional inequality 
    \begin{align} \label{fctl_ineq_phi_Linfty}
        r \phi(s+r) \leq C' \phi (s)^{1+\delta_0}
    \end{align}
    where $\delta_0 \coloneqq (p-k)/(kp) > 0$ for a $p>k$ to be chosen later and $C' >0$ depends on the same quantities as $C$. We prove first that the energy of the $\varphi_t$ is uniformly bounded under our assumptions, then that a DeGiorgi-type iteration yields the Theorem if (\ref{fctl_ineq_phi_Linfty}) holds.

    \begin{Lemma}
        There is a constant $C_E>0$ depending on the same quantities as allowed in Theorem \ref{Thm_HEq_Linftyestimate} such that
        \begin{align*}
            E(\varphi_t) \coloneqq \int_X -\varphi_t e^{F_t} \omega_X^n \leq C_E.
        \end{align*}
    \end{Lemma}

    \begin{proof}
        First, fix a uniform constant $C>0$ such that $\omega_t \leq C\omega_X$.
        We apply Hölder's inequality with $q$ and its dual exponent $\frac{q}{q-1}$. Note that because of $q>2$ and $k=n-1$ we have  $\frac{q}{q-1}< n$ such that there is uniform $L^\frac{q}{q-1}$ bound for $(C\omega_X, \omega_X, k)$-sh functions by Proposition \ref{ksh_Lp_bound}. Since all $\varphi_t$ are $(\omega_t, \omega_X, k)$-sh, they are in particular $(C\omega_X, \omega_X, k)$-sh. We get
        \begin{align*}
            E(\varphi_t) \leq \left( \int_X |\varphi_t|^\frac{q}{q-1} \omega_X^n\right)^\frac{q-1}{q} \left( \int_X e^{qF_t} \omega_X^n \right)^\frac{1}{q}\leq N^\frac{1}{q} C(C\omega_X, n, k). 
        \end{align*}
    \end{proof}

    \subsection{DeGiorgi-type iteration} \label{subsec_DeGiorgi}
        We first note that $\phi$ is not-increasing and decays at least at rate of $1/s$. Indeed, on $\Omega_s$ for $s>0$ we have $-\varphi_t/s >1$ so that
        \begin{align*}
            \phi(s) < \int_{\Omega_s} \frac{-\varphi_t}{s} e^{F_t} \omega_X^n \leq \frac{1}{s} \int_{\Omega_0} -\varphi_t e^{F_t} \omega_X^n \leq  \frac{E(\varphi_t)}{ s} \leq \frac{C_E}{s}. 
        \end{align*}
        We choose $s_0 \coloneqq C_E(2C')^{1/\delta_0}$ such that 
        \begin{align*}
            \phi (s_0)^{\delta_0} < \frac{1}{2C'}.
        \end{align*}
        Moreover, for $i \in \{0,1,2,\dots \}$ we define $r_i \coloneqq 2^{-i\delta_0}$ and we iterate (\ref{fctl_ineq_phi_Linfty}) starting from $s_0$ with the $r_i$'s. We set
        \begin{align*}
            C \coloneqq s_0 + \sum_{i = 0}^{\infty} r_i = (2C')^{1/\delta_0} + \frac{1}{1-2^{-\delta_0}}
        \end{align*}
        which only depends on the quantities stated in Theorem \ref{Thm_HEq_Linftyestimate}.

        \textbf{Claim.} For every $m \geq 0$ we have
        \begin{align*}
            \phi\left(s_0 + \sum_{i=0}^m 2^{-i\delta_0} \right) \leq C'^{-1/\delta_0} 2^{-m-1-1/\delta_0}.
        \end{align*}

        If the claim holds, $\phi(s) = 0$ if $s \geq C$ so Theorem \ref{Thm_HEq_Linftyestimate} is established.

        \begin{proof}
            The proof proceeds by induction on $m$. If $m=0$ we use (\ref{fctl_ineq_phi_Linfty}) to get
            \begin{align*}
                \phi (s_0 + 2^{-0\delta_0}) \leq 2^{0\delta_0} C' \phi(s_0)^{1+\delta_0} \leq 2^{0\delta_0} C' (2C')^{-1 - 1/\delta_0} \leq C'^{-1/\delta_0} 2^{-0-1-1/\delta_0}
            \end{align*}
            by the choice of $s_0$.\\
            Now, assume the claim holds true for $m \geq 0$. Then for $m+1$ using (\ref{fctl_ineq_phi_Linfty}) again we get
            \begin{align*}
                \phi\left(s_0 + \sum_{i=0}^{m+1} 2^{-i\delta_0} \right) &\leq 2^{(m+1)\delta_0} C' \phi\left(s_0 + \sum_{i=0}^m 2^{-i\delta_0} \right)^{1+\delta_0}\\
                &\leq 2^{(m+1)\delta_0} C' \left(C'^{-1/\delta_0} 2^{-m-1-1/\delta_0}\right)^{1+\delta_0}\\
                &= 2^{(m+1)\delta_0} C'^{1 -1/\delta_0 -1} 2^{-m-1-1/\delta_0 -m\delta_0 -1\delta_0 -1}\\
                &= C'^{ -1/\delta_0 } 2^{-(m+1)-1-1/\delta_0 },
            \end{align*}
            where in the second inequality we apply the induction hypothesis.
        \end{proof}

    \subsection{Auxiliary Hessian equation} \label{Sec_auxHEq_Linfty}

    We will estimate the modified energy
    \begin{align*}
        A_s \coloneqq \frac{1}{V_{t,k}}\int_{\Omega_s} (-\varphi_t-s) e^{F_t} \omega_X^n
    \end{align*}
    corresponding to $\Omega_s$ to yield \ref{fctl_ineq_phi_Linfty}.
    
    Fix a sequence of smooth, positive functions $\eta_l(-\varphi_t-s)$ that converge uniformly to $\max\{0, -\varphi_t-s\}$ from above for $l\rightarrow\infty$. 

    Consider the auxiliary Hessian equation
    \begin{align} \label{AuxHEq_Linftyestimate}
        (\omega_t + dd^c \psi_l)^k \wedge \omega_X^{n-k} = \frac{1}{V_{t,k}}\frac{\eta_l(-\varphi_t -s)}{A_{s,l}} e^{F_t} \omega_X^n, \quad \sup_X \psi_l = 0
    \end{align}
    where $A_{s,l}$ is the normalizing constant
    \begin{align*}
        A_{s,l} \coloneqq \frac{1}{V_{t,k}}\int_{X} \eta_l(-\varphi_t-s) e^{F_t} \omega_X^n
    \end{align*}
    that converges to $A_s$ from above by Lebesgue's dominated convergence theorem. By Proposition \ref{resolution_result_HEq}, it admits an unique smooth solution because $\omega_t$ and $\omega_X$ are Kähler, the RHS is smooth and positive, and the equation is suitably normalized.

    Next, we would like to compare the solution $\psi_l$ of our auxiliary equation to $(-\varphi_t -s)$. To this extent, define for some positive constants $\varepsilon, \Lambda>0$ the smooth comparison function
    \begin{align*}
        \Phi \coloneqq -\varepsilon(-\psi_l + \Lambda)^{\frac{k}{k+1}} + (-\varphi_t - s).
    \end{align*}
    We will show that with the right choice of $\varepsilon$ and $\Lambda$ we have $\Phi \leq 0$. For this, consider the maximum of $\Phi$ which is attained at some point $x_0$. In case that $x_0 \notin \Omega_s$ we have that $\varphi_t > -s$ so the second term of $\Phi$ is negative. The first one is also negative by normalization of $\psi_l$. In the other case, we apply the maximum principle with the linearization of the Hessian operator at $\omega_t+dd^c\varphi_t$. It is an elliptic operator because $\varphi_t$ is $(\omega_t, \omega_X, k)$-sh. This yields
    \begin{align*}
        0 &\geq L^k_{\omega_{t,\varphi_t}, \omega_X} \Phi |_{x_0} \\
        &= \frac{dd^c(- \varepsilon (-\psi_{l} + \Lambda )^\frac{k}{k+1} -\varphi_t - s) \wedge (\omega_t+dd^c\varphi_t)^{k-1} \wedge \omega_X^{n-k} }{ e^{F_t}\omega_X^{n}}\\
        &\geq \frac{k\varepsilon}{k+1} (-\psi_{l} + \Lambda )^\frac{-1}{k+1} \left ( \frac{(\omega_t + dd^c \psi_{l})\wedge (\omega_t+dd^c\varphi_t)^{k-1} \wedge \omega_X^{n-k} }{e^{F_t}\omega_X^{n}} -  \frac{(\omega_t+dd^c\varphi_t)^k \wedge \omega_X^{n-k} } {e^{F_t}\omega_X^{n}}\right) \\ & \quad \quad - \frac{(\omega_t + dd^c\varphi_t ) \wedge (\omega_t+dd^c\varphi_t)^{k-1} \wedge \omega_X^{n-k} }{ e^{F_t}\omega_X^{n}} + \frac{(\omega_t+dd^c\varphi_t)^k \wedge \omega_X^{n-k}}{e^{F_t}\omega_X^{n}}.
    \end{align*}
    Next, we apply the mixed Hessian inequality (see Proposition \ref{mixed_hessian_ineq}) and the equations solved by $\psi_l$ and $\varphi_t$. For this, consider a local chart $g: U \longrightarrow X$ around $x_0$ and let $\beta\coloneqq g^*\omega_X$ be the pull-back of $\omega_X$ to this chart as well as $v$ a local potential of $g^*\omega_t$. Let $u_1 \coloneqq v + \psi_l \circ g$ and for $2\leq l\leq k$ let $u_l\coloneqq v + \varphi_t$. Then, in the context of Proposition \ref{mixed_hessian_ineq}, $f_1 = \left(\frac{\eta_l(-\varphi_t - s) }{V_{t,k} A_{s,l}}e^{F_t}\right)\circ g$ and $f_l = e^{F_t}\circ g$ so that $dd^cu_1 \wedge \dotso \wedge dd^cu_k \wedge \beta^{n-k} \geq (f_1 \cdot \dotso \cdot f_k)^{1/k}\beta^n$ which translates on $X$ to
    \begin{align*}
        (\omega_t + dd^c \psi_{l})\wedge (\omega_t+dd^c\varphi_t)^{k-1} \wedge \omega_X^{n-k}  \geq \left( \frac{\eta_l(-\varphi_t - s) }{V_{t,k} A_{s,l}} \right)^\frac{1}{k} e^{F_t} \omega_X^n.
    \end{align*}
    This yields
    \begin{align*}
        0 &\geq \frac{k\varepsilon}{k+1} (-\psi_{l} + \Lambda )^\frac{-1}{k+1}  \left ( \frac{1}{V_{t,k}}\frac{\eta_l(-\varphi_t -s)}{A_{s,l}}  \right)^{1/k}  - 1 \\ & \quad \quad + \left(1- \frac{k\varepsilon}{k+1} (-\psi_{l} + \Lambda )^\frac{-1}{k+1} \right).
    \end{align*}
    If the constants are chosen as 
    \begin{align*}
        0 &< \varepsilon \coloneqq \left( \frac{k+1}{k} \right)^{\frac{k}{k+1}}(V_{t,k}A_{s,l})^{\frac{1}{k+1}},\\
        0 &< \Lambda \coloneqq \frac{k}{k+1}(V_{t,k}A_{s,l})
    \end{align*}
    then 
    \begin{align*}
        \frac{k\varepsilon}{k+1} (-\psi_{l} + \Lambda )^\frac{-1}{k+1} &= \frac{k}{k+1} \left( \frac{k+1}{k} \right)^{\frac{k}{k+1}}(V_{t,k}A_{s,l})^{\frac{1}{k+1}} \left( -\psi_{l} + \frac{k}{k+1}(V_{t,k}A_{s,l})\right)^\frac{-1}{k+1}\\
        &\leq  \frac{k}{k+1} \left( \frac{k+1}{k} \right)^{\frac{k}{k+1}} (k+1)^{\frac{1}{k+1}}k^{\frac{-1}{k+1}}\\
        &\leq 1,
    \end{align*}
    so we can drop the the term $\left(1- \frac{k\varepsilon}{k+1} (-\psi_{l} + \Lambda )^\frac{-1}{k+1} \right)$ to find 
    \begin{align*}
        \frac{k\varepsilon}{k+1} (-\psi_{l} + \Lambda )^\frac{-1}{k+1}  \left ( \frac{1}{V_{t,k}}\frac{\eta_l(-\varphi_t -s)}{A_{s,l}} \right)^{1/k} \leq 1
    \end{align*}
    or
    \begin{align*}
         \frac{1}{V_{t,k}}\frac{\eta_l(-\varphi_t -s)}{A_{s,l}}  \leq \left(\frac{k+1}{k}\right)^k\varepsilon^{-k}  (-\psi_{l} + \Lambda )^\frac{k}{k+1}.
    \end{align*}
    By the definition of $\varepsilon$ and because $\eta_l(-\varphi_t -s)$ was chosen to decrease to $-\varphi_t -s$ we get
    \begin{align*}
        -\varphi_t-s \leq \varepsilon (-\psi_{l} + \Lambda )^\frac{k}{k+1}
    \end{align*}
    or in other words, the comparison function $\Phi$ is non-positive, as claimed.

    \begin{Bemerkung}
        To this point, $\varepsilon$ and $\Lambda$ depend on $t$ (and $l$). It should be noted however, that for $t \in ]0,1]$, the intersection number $V_{t,k}$ is bounded from below and above. The dependency on $A_{s,l}$ will be made explicit again in the next equation.
    \end{Bemerkung}

    \subsection{Application of the uniform $L^p$ bound}
    Because $\Phi \leq 0$, on $\Omega_s$ we have
    \begin{align*}
        \left(\frac{-\varphi_t-s}{A_{s,l}^{\frac{1}{k+1}}}\right)^\frac{k+1}{k} \leq C_1 (-\psi_l + C_2A_{s,l})
    \end{align*}
    which enables us to apply the uniform $L^p$-norm bound for $(C_3\omega_X, \omega_X)$-$k$-sh functions (cf.\ Proposition \ref{ksh_Lp_bound}). Similarly as above we get for $p\in ]0, \frac{n}{n-k}[$ that
        \begin{align*}
            \int_{\Omega_s} \left( (-\varphi_t-s)^\frac{k+1}{k} A_{s,l}^{-\frac{1}{k}} \right)^p \omega_X^n  &\leq \int_{\Omega_s} \left( C_1 (-\psi_{t,k} + \Lambda )\right)^p \omega_X^n\\
            &\leq C_4(p)(1+ A_{s,l}^p).
        \end{align*}

    By taking the limit $l\rightarrow \infty$ now and also observing that for all $s\geq 0$ we have $A_s \leq E(\varphi_t) = \int_X (-\varphi_t)e^{F_t}\omega_X^n$ we arrive at

    \begin{Lemma} \label{estimate_modified_energy}
        For every $p\in ]0, \frac{n}{n-k}[$ there is a constant $C_4(p)$ independent of $t$ such that
         \begin{align*}
            \int_{\Omega_s} \left( (-\varphi_t-s)^\frac{k+1}{k} A_{s}^{-\frac{1}{k}} \right)^p \omega_X^n 
            &\leq C_4(p)(1+ E(\varphi_t)^p).
        \end{align*}
    \end{Lemma}

    Note the structural similarity to Lemma 1 of \cite{guo$L^infty$EstimatesComplex2021} where instead of $p$-th power the exponential is used. This is due to the better integrability of psh functions compared to the $k$-sh ones we work with here. 

    Now, let us fix $p$ such that $k<p< p + \frac{1}{q-1}< \frac{n}{n-k}$ (recall here that we assumed $k=n-1$ and $q>2$).

    We use Young's inequality to estimate
        \begin{align*}
            \left( (-\varphi_t-s)^\frac{k+1}{k} A_{s}^{-\frac{1}{k}} \right)^p e^{F_t} \leq e^{qF_t} + \left( (-\varphi_t-s)^\frac{k+1}{k} A_{s}^{-\frac{1}{k}} \right)^{p+\frac{1}{q-1}}
        \end{align*}
        on $\Omega_s$.
        Therefore,
        \begin{align*}
            \int_{\Omega_s} \left( (-\varphi_t-s)^\frac{k+1}{k} A_{s}^{-\frac{1}{k}} \right)^p e^{F_t} \omega_X^n \leq N + C_4\left(p+\frac{1}{q-1}\right) \left(1+ E(\varphi_t)^{p+\frac{1}{q-1}}\right).
        \end{align*}
        
         By dividing this inequality by positive constants we get
        \begin{align*}
            \int_{\Omega_s} (v-s)^\frac{p(k+1)}{k}e^{F_t} \omega_X^n \leq  A_{s}^{p/k} \left( N + C_4\left(p+\frac{1}{q-1}\right) \left(1+ E(\varphi_t)^{p+\frac{1}{q-1}}\right)\right).
    \end{align*}
    We use this estimate after applying the Hölder inequality to derive
    \begin{align*}
        A_s &= \frac{1}{V_{t,k}}\int_{\Omega_s} (-\varphi_t-s) e^{F_t} \omega_X^n\\
        &\leq \frac{1}{V_{t,k}} \left(\int_{\Omega_s} (v-s)^\frac{p(k+1)}{k}e^{F_t} \omega_X^n \right)^{\frac{k}{p(k+1)}} \left( \int_{\Omega_s} e^{F_t} \omega_X^n \right)^{\frac{1}{p'}}\\
        &\leq \frac{1}{V_{t,k}} \left( A_{s}^{p/k} \left( N + C_4\left(p+\frac{1}{q-1}\right) \left(1+ E(\varphi_t)^{p+\frac{1}{q-1}}\right)\right) \right)^{\frac{k}{p(k+1)}} \left( \int_{\Omega_s} e^{F_t} \omega_X^n \right)^{\frac{1}{p'}}
    \end{align*}
    which can be reformulated -- using that the dual exponent $p' = \frac{p(k+1)}{p(k+1) - k}$, the energy is bounded by $C_E$, and the definition of $\phi(s)$ -- as
    \begin{align*}
        A_s &\leq \left(\frac{1}{V_{t,k}} \right)^{\frac{k+1}{k}}\left( N + C_4\left(p+\frac{1}{q-1}\right) \left(1+ E(\varphi_t)^{p+\frac{1}{q-1}}\right)\right)^{\frac{1}{p}} \left( \int_{\Omega_s} e^{F_t} \omega_X^n \right)^{\frac{pk + p - k}{pk}}\\
        &\leq C_5 \phi(s)^{1+\delta_0}
    \end{align*}
    implying (\ref{fctl_ineq_phi_Linfty}) for the De Giorgi iteration if $p>k$. 

    In other words, either if $k = n-1$ or the uniform integrability exponent in Proposition \ref{ksh_Lp_bound} can be improved upon, the desired $L^\infty$ estimate follows.

\section{Proof of the mean-value inequality}

First, we fix $0 < t \leq 1$ and $v\in L^1(X, \omega_t^k \wedge \omega_X^{n-k})$. We may assume that $v$ is not constantly zero and satisfies $\|v\|_{L^1} \leq V_{k,0} = [\chi]^k\wedge [\omega_X]^{n-k}  $. If not, we replace it by $V_{k,0} v/\|v\|_{L^1}$ which is still $\Cont^2$ wherever non-negative, has mean-value zero on $X$ with respect to $\omega_t^k\wedge \omega_X^{n-k}$, and its mixed Laplacian is bounded from below by $-a$. Recall that we normalized the fixed reference metric $\omega_X$ to have volume $V=1$ for simplicity.

    We need to show that $\sup_X v \leq C$ where $C$ is independent of $v$ and $t$. Similarly to the previous section, this is equivalent to proving that the function
    \begin{align*}
        \phi \colon \RR_{\geq 0} \longrightarrow \RR_{\geq 0}, \quad \quad s \longmapsto \frac{1}{V_{k,t}} \int_{\Omega_s} \omega_t^k\wedge \omega_X^{n-k} = \int_{\Omega_s} e^{F_t} \omega_X^n
    \end{align*}
    vanishes for $s \geq C$. Here and in the following, $\Omega_s = \{v > s\}$ denotes the upper-level set of $v$. Again, the main part of the proof is to establish for every $s\geq 0$ and every $r \geq 0$ a functional inequality 
    \begin{align} \label{fctl_ineq_phi}
        r \phi(s+r) \leq C' \phi (s)^{1+\delta_0}
    \end{align}
    where $\delta_0 \coloneqq (p-k)/(kp) > 0$ for a $p>k$ to be chosen later and $C' >0$ depends on the same quantities as $C$. A De Giorgi-type iteration yields the Theorem if (\ref{fctl_ineq_phi}) holds (see Section \ref{subsec_DeGiorgi}):

        We first note that $\phi$ is not-increasing and decays at least at rate of $1/s$. Indeed, on $\Omega_s$ for $s>0$ we have $v/s >1$ so that
        \begin{align*}
            \phi(s) < \int_{\Omega_s} \frac{v}{s} e^{F_t} \omega_X^n \leq \frac{1}{s} \int_{\Omega_0} v e^{F_t} \omega_X^n \leq  \frac{\|v\|_{L^1(X , \omega_t^k \wedge \omega_X^{n-k})}}{V_t s} < \frac{1}{s}
        \end{align*}
        by the assumption $\|v\|_{L^1(X , \omega_t^k \wedge \omega_X^{n-k})} \leq V_{k,0} < V_{k,t}$. We choose $s_0 \coloneqq (2C')^{1/\delta_0}$ such that 
        \begin{align*}
            \phi (s_0)^{\delta_0} < \frac{1}{2C'}.
        \end{align*}
        Moreover, for $i \in \{0,1,2,\dots \}$ we define $r_i \coloneqq 2^{-i\delta_0}$ and we iterate (\ref{fctl_ineq_phi}) starting from $s_0$ with the $r_i$'s. We set
        \begin{align*}
            C \coloneqq s_0 + \sum_{i = 0}^{\infty} r_i = (2C')^{1/\delta_0} + \frac{1}{1-2^{-\delta_0}}
        \end{align*}
        which only depends on the quantities stated in Theorem \ref{meanvalueineq}.

    \subsection{Auxiliary Hessian equation and application} \label{subsec_auxMA_meanvalue}
        To prove the functional inequality (\ref{fctl_ineq_phi}), we will setup an auxiliary Hessian equation and show that a certain function $\Phi$ depending on its solution is non-positive. This allows to estimate quantities related to (\ref{fctl_ineq_phi}) against the obtained solution, which is $(C\omega_X, \omega_X, k)$-sh such that the uniform $L^p$-bound for $k$-sh functions (cf.\ Proposition \ref{ksh_Lp_bound}) leads to an uniform upper bound. 

        We fix $s\geq 0$ and assume that $\Omega_s \neq \emptyset$ since $\phi(s) = 0$ otherwise. For any $r \geq 0$ we find that
        \begin{align*}
            r\phi(s+r) = (r + s -s ) \int_{\Omega_{s+r}} e^{F_t} \omega_X^n \leq \int_{\Omega_{s}} (v-s) e^{F_t} \omega_X^n \eqqcolon A_s
        \end{align*}
        so that it is enough to estimate $A_s$ to obtain (\ref{fctl_ineq_phi}). We remark that 
        \begin{align*}
            A_s \leq \frac{1}{V_{k,t}} \int_{\Omega_{s}} |v|  \omega_t^{k} \wedge \omega_X^{n-k} \leq \frac{\|v\|_{L^1(X , \omega_t^k \wedge \omega_X^{n-k})}}{V_{k,t}} \leq 1
        \end{align*}
        by the assumption on $v$. 
        
        For the setup of the Hessian equation we need to cut-off the function $v-s$ at $0$ and smooth it to obtain a smooth solution by Proposition \ref{resolution_result_HEq}. For this, fix a sequence of smooth, positive functions $\eta_l(v-s) \colon X \longrightarrow \RR$ converging uniformly to $\max \{0, v-s\}$ from above for $l \rightarrow \infty$.

        It follows by dominated convergence that we can approximate $A_s$ by 
        \begin{align*}
            A_{s,l} \coloneqq \int_{\Omega_{s}} \eta_l(v-s) e^{F_t} \omega_X^n,
        \end{align*}
        in particular, $A_{s,l}\leq 2$ for $l$ large enough.

        We consider the auxiliary Hessian equation
        \begin{align} \label{aux_Heq_mvineq}
            (\omega_t + d d^c \psi_l )^k \wedge \omega_X^{n-k} = V_{k,t}\frac{\eta_l(v-s)}{A_{s,l} }e^{F_t}\omega_X^n, \quad \quad \sup_X \psi_l = 0.
        \end{align}
        It admits an smooth solution by Proposition \ref{resolution_result_HEq} because $\omega_t$ is a Kähler form and the RHS is smooth, positive and its integral over $X$ is equal to $V_{k,t}$. 
        
        In the following, we aim to obtain an estimate for (powers of) $v-s$ and $A_{s,l}$ in terms of the solution $\psi_l$ and constants. Since $\psi_l$ is $(C\omega_X, \omega_X, k)$-sh, we can then apply the $L^p$-bound. This replaces the $\alpha$-invariant estimate in the case of degenerations to big classes.
        
        For this, we need to fix a constant: we claim that there is $\varepsilon > 0$ such that 
        \begin{align} \label{def_vareps_mvineq}
            \varepsilon^{k+1} = \left( \frac{k+1}{kC(n,k)}\right)^k (a + n\varepsilon)^k A_{s,l}.  
        \end{align}
        To see this, define a polynomial $f \colon \RR \longrightarrow \RR$ by
        \begin{align*}
            f(x) &\coloneqq x^{k+1} - \left( \frac{k+1}{kC(n,k)}\right)^n (a + nx)^k A_{s,l} \\
            &= x^{k+1} - \sum_{i=0}^k A_{s,l} \left( \frac{k+1}{kC(n,k)}\right)^k \binom{k}{i} a^{k-i} n^i x^i.
        \end{align*}
        Then it is clear that $f(0) < 0$ and $\lim_{x\rightarrow \infty} f(x) = + \infty$, so by the intermediate value theorem $f$ has a positive root $\varepsilon$. Moreover, the Lagrange-Zassenhaus bound for polynomial roots yields
        \begin{align*}
            \varepsilon &\leq 2 \max_{i \in \{0, \dots , k\} } \left( A_{s,l} \left( \frac{k+1}{kC(n,k)}\right)^k \binom{k}{i} a^{k-i} n^i \right)^{\frac{1}{k-i+1}}\\
            &\leq 2 \max_{i \in \{0, \dots , k\} } \left( 2^{(\frac{1}{k-i+1}-\frac{1}{k+1})} \left( \frac{k+1}{kC(n,k)}\right)^k \binom{n}{i} a^{k-i} n^i \right)^{\frac{1}{k-i+1}} A_{s,l}^\frac{1}{k+1}
        \end{align*}
        since $A_{s,l}\leq 2$ so that we have 
        \begin{align} \label{estimate_varpeps}
            \varepsilon &\leq C(n,k,a) A_{s,l}^\frac{1}{k+1}
        \end{align}
        for a constant depending only on $a$ and $n$.

        \textbf{Claim.} The function $\Phi \colon X \longrightarrow \RR$, 
        \begin{align*}
            \Phi \coloneqq - \varepsilon (-\psi_{l} + 1 )^\frac{k}{k+1} + v - s 
        \end{align*}
        is non-positive.

        \begin{Bemerkung}
            The function $\Phi$ is chosen in such a way that when computing its mixed Laplacian, $\psi_l$ as well as $ L^k_{\omega_t, \omega_X} \psi_l$ appear. The latter can be estimated using the auxiliary Hessian equation (\ref{aux_Heq_mvineq}) allowing to compare the RHS of (\ref{aux_Heq_mvineq}) and its $(\omega_t,\omega_X,k)$-sh solution $\psi_l$ directly.
        \end{Bemerkung}

        \begin{proof}
            First, we note that the parenthesis in the first term of $\Phi$ is positive by the normalization $\sup \psi_l = 0$.
            On $X\setminus \Omega_s$ we have $v \leq s$ and the first term of $\Phi$ is negative, so $\Phi < 0$ there. On the other hand, $\Phi$ is twice continuously differentiable on $\overline{\Omega_s}$ and for the Claim to fail it has to achieve its maximum therefore at a point $x_0 \in \Omega_s$.

            We apply the maximum principle: again, because $\omega_t$ and $\omega_X$ are Kähler metrics, the linearized Hessian operator $L^k_{\omega_t, \omega_X}$ is elliptic. The maximum principle implies 
            \begin{align*}
                L^k_{\omega_t, \omega_X} \Phi \leq 0
            \end{align*}
            where the LHS can be estimated 
            \begin{align*}
                L^k_{\omega_t, \omega_X} \Phi &= \frac{dd^c(- \varepsilon (-\psi_{l} + 1 )^\frac{k}{k+1} + v - s) \wedge \omega_t^{k-1} \wedge \omega_X^{n-k} }{ \omega_t^k \wedge \omega_X^{n-k}}\\
                &\geq \frac{k\varepsilon}{k+1} (-\psi_{t,l} + 1 )^\frac{-1}{k+1} \left ( \frac{(\omega_t + dd^c \psi_{t,l})\wedge \omega_t^{k-1} \wedge \omega_X^{n-k} }{\omega_t^k \wedge \omega_X^{n-k}} - n \right) - a.
            \end{align*}
            Because $\omega_t^{k-1} \wedge \omega_X^{n-k}$ is a positive form, we could neglect the term containing only first order derivatives in the above computation. 
            
            Next, we apply as in Section \ref{Sec_auxHEq_Linfty} the mixed Hessian inequality (see Proposition \ref{mixed_hessian_ineq}) and the equation (\ref{aux_Heq_mvineq}) solved by $\psi_l$. Let $u_1$ be in a local chart a potential of $\omega_t + dd^c \psi_{l}$, and $u_2$ a local potential of $\omega_t$, as well as $\beta$ the pull-back of $\omega_X$. Then, by abuse of notation identifying functions with their pull-backs in the coordinates, we have
            \begin{align*}
                dd^cu_1 \wedge dd^cu_2 \wedge \dotso \wedge dd^cu_2 \wedge \beta^{n-k} \geq \left( V_{k,t} \frac{\eta_l(v-s)}{A_{s,l}} e^{F_t} \cdot V_{k,t}^{k-1}e^{(k-1)F_t}\right)^{1/k} \beta^{n}
            \end{align*}
            in other words on $X$
            \begin{align*}
                \frac{(\omega_t + dd^c \psi_{l})\wedge \omega_t^{k-1} \wedge \omega_X^{n-k} }{\omega_t^k \wedge \omega_X^{n-k}} \geq \left(  \frac{\eta_l(v-s)}{A_{s,l}} \right)^{1/k}
            \end{align*}
            holds. Thus, we get the estimate 
            \begin{align*}
                0 &\geq L^k_{\omega_t, \omega_X} \Phi \\
                &\geq \frac{k\varepsilon}{k+1} (-\psi_{l} + 1 )^\frac{-1}{k+1} \left (C(n,k) \left( \frac{(\omega_t + dd^c \psi_{l})^k \wedge \omega_X^{n-k} }{\omega_t^k \wedge \omega_X^{n-k}}\right)^\frac{1}{k} - n \right) - a\\
                &= \frac{k\varepsilon}{k+1} (-\psi_{l} + 1 )^\frac{-1}{k+1} \left (C(n,k) \left( \frac{\eta_l(v-s)}{A_{s,l}}\right)^\frac{1}{k} - n \right) - a.
            \end{align*} 
            
            This simplifies -- keeping in mind that $(-\psi_{l} + 1 )^\frac{-1}{k+1} \leq 1$ and $\eta_l(v-s) \geq v-s$ by construction of the approximation -- to
            \begin{align*}
                a + \varepsilon n \geq \frac{\varepsilon k C(n,k)}{k+1} (-\psi_{l} +1 )^\frac{-1}{k+1} \left( \frac{v-s}{A_{s,l}} \right)^{\frac{1}{k}}
            \end{align*}
            which holds at $x_0$. We take the $k$-th power keeping in mind that both sides are semi-positive and divide by positive constants to get
            \begin{align*}
                \left( \frac{k+1}{kC(n,k)}\right)^k (a + \varepsilon n)^k A_{s,l} (-\psi_{l} + 1)^\frac{k}{k+1} \geq \varepsilon^k (v-s).
            \end{align*}
            By the definition of $\varepsilon$ in (\ref{def_vareps_mvineq}) this is equivalent to
            \begin{align*}
                \varepsilon (-\psi_{k} + 1 )^\frac{k}{k+1} \geq v -s
            \end{align*}
            at $x_0$. In other words, $\max_X \Phi \leq 0$, the claim is established.
        \end{proof}

    \subsection{Application of the uniform $L^p$ bound}    \label{subsec_alphainv_application}
        The above claim together with the estimate (\ref{estimate_varpeps}) for $\varepsilon$ gives
        \begin{align*}
            v-s \leq C(n,k,a) A_{s,l}^{\frac{1}{k+1}} (-\psi_{l} + 1)^\frac{k}{k+1}
        \end{align*}
        and on $\Omega_s$ where the LHS is non-negative we can take the $\frac{k+1}{k}$-th power of this inequality so that
        \begin{align} \label{ineq_of_integrands2}
            (v-s)^\frac{k+1}{k} A_{s,l}^{-\frac{1}{k}} \leq C_1 (-\psi_{l} + 1 ).
        \end{align}
        
        Instead of applying the Hörmander-Tian $\alpha$-invariant estimate for $\omega_X$-psh functions, we use the $L^p$ bound for $(C_3 \omega_X,\omega_X,k)$-sh functions as found in \cite{guo$L^infty$EstimatesComplex2021}: Since there is a uniform constant $C_3 > 0$ such that for all $t \in ]0,1]$ we have $\omega_t \leq C_3 \omega_X$, all the $\psi_{l}$ are $(C_3 \omega_X,\omega_X, k)$-sh independent of $t$. Then Proposition \ref{ksh_Lp_bound} implies for $p\in ]0, \frac{n}{n-k}[$ that
        \begin{align*}
            \int_{\Omega_s} \left( (v-s)^\frac{k+1}{k} A_{s,k}^{-\frac{1}{k}} \right)^p \omega_X^n  &\leq \int_{\Omega_s} \left( C_1 (-\psi_{l} + 1 )\right)^p \omega_X^n\\
            &\leq C_2(p).
        \end{align*}
        Now, let us fix $p$ such that $k<p< p + \frac{1}{q-1}< \frac{n}{n-k}$ (recall here that we assumed $k=n-1$ and $q>2$). By Young's inequality, 
        \begin{align*}
            \left( (v-s)^\frac{k+1}{k} A_{s,k}^{-\frac{1}{k}} \right)^p e^{F_t} \leq e^{qF_t} + \left( (v-s)^\frac{k+1}{k} A_{s,l}^{-\frac{1}{k}} \right)^{p+\frac{1}{q-1}}
        \end{align*}
        holds on $\Omega_s$.
        Therefore,
        \begin{align*}
            \int_{\Omega_s} \left( (v-s)^\frac{k+1}{k} A_{s,l}^{-\frac{1}{k}} \right)^p e^{F_t} \omega_X^n \leq N + C_2\left(p+\frac{1}{q-1}\right) 
        \end{align*}
        
         By dividing this inequality by positive constants we get
        \begin{align*}
            \int_{\Omega_s} (v-s)^\frac{p(k+1)}{k}e^{F_t} \omega_X^n \leq C_4 A_{s,l}^{p/k} 
        \end{align*}
        where the RHS converges to $C_4 A_s^{p/k} $ for $l \rightarrow \infty$ but the LHS is independent of $l$.

        As remarked above, we want to estimate $A_s$ which we can do now by applying Hölder's inequality with exponent $\frac{p(k+1)}{k}$ and its dual exponent $p'$ to get
        \begin{align*}
            A_s = \int_{\Omega_s} (v-s) e^{F_t} \omega_X^n &\leq \left(\int_{\Omega_s} (v-s)^\frac{p(k+1)}{k}e^{F_t} \omega_X^n \right)^{\frac{p(k+1)}{k}} \left( \int_{\Omega_s} e^{F_t} \omega_X^n \right)^{\frac{1}{p'}}\\ 
            &\leq \left(C_4 A_s^{p/n} \right)^{\frac{p(k+1)}{k}}  \left( \int_{\Omega_s} e^{F_t} \omega_X^n \right)^{\frac{1}{p'}} \\
            &= C_5 A_s^{\frac{1}{k+1}} \phi(s)^{\frac{1}{p'}}
        \end{align*}
        and since we did assume $A_s > 0$ this leads to
        \begin{align*}
            r\phi(s+r) \leq A_s \leq C' \phi(s)^{\frac{k+1}{kp'}}
        \end{align*}
        which is equivalent to (\ref{fctl_ineq_phi}) since $\frac{k+1}{kp'} = \frac{pk +p -k}{kp} = 1 + \frac{p-k}{kp} \eqqcolon 1 + \delta_0$ given that $k < p < \frac{n}{n-k}$. This finishes the proof of Theorem \ref{meanvalueineq}.











\printbibliography[title={Bibliography}]

\pagestyle{empty}

\pagestyle{empty}
\vspace*{\fill}
\cleardoublepage

\end{document}